\documentclass[12pt]{amsart}
\usepackage{amssymb}
\usepackage[utf8]{inputenc}
\usepackage{t1enc}
\usepackage{graphicx}
\usepackage{indentfirst}

\newcommand{\nn}{\mathbb{N}}

\newcommand\norm[1]{\left\lVert#1\right\rVert} 
\newcommand\dotprod[1]{\langle#1\rangle}       

\newtheorem{Def}{Definition}
\newtheorem{Lem}{Lemma}

\newtheorem{Thm}{Theorem}
\newtheorem{Cor}{Corollary}
\newtheorem{Rem}{Remark}

\newcommand*\colvec[1]{\begin{pmatrix}#1\end{pmatrix}}

\title[On generalized Berwald manifolds of dimension three]{On generalized Berwald manifolds of dimension three}
\author{Csaba Vincze}
\address{Institute of Mathematics, University of Debrecen, H-4002 Debrecen, P.O.Box 400, Hungary}
\email{csvincze@science.unideb.hu}
\author{Márk Oláh}
\address{Institute of Mathematics, University of Debrecen, Doctoral School of Mathematical and Computational Sciences, H-4002 Debrecen, P.O.Box 400, Hungary}
\email{olma4000@gmail.com}

\begin{document}

\begin{abstract}

A linear connection on a Finsler manifold is called compatible to the Finsler function if its parallel transports preserve the Finslerian length of tangent vectors. Generalized Berwald manifolds are Finsler manifolds equipped with a compatible linear connection. In the paper we present a general and intrinsic method to characterize the compatible linear connections on a Finsler manifold of dimension three. We prove that if a compatible linear connection is not unique then the indicatrices must be Euclidean surfaces of revolution. The surplus freedom of choosing compatible linear connections is related to Euclidean symmetries. The unicity of the solution of the compatibility equations can be provided by some additional requirements. Following the idea in \cite{V14} we are also looking for the so-called extremal compatible linear connection minimizing the norm of its torsion at each point of the manifold.  
\end{abstract}

\maketitle
\footnotetext[1]{Cs. Vincze is supported by the EFOP-3.6.1-16-2016-00022 project. The project is co-financed by the European Union and the European Social Fund.}
\footnotetext[2]{Cs. Vincze and M. Oláh are also supported by TKA-DAAD 307818.}

\section{Notation and terminology}

Let $M$ be a smooth manifold with a local coordinate system ($u^1, \ldots, u^n$). The induced local coordinate system on the tangent manifold $TM$ consists of the functions $x^1, \ldots, x^n$ and $y^1, \ldots, y^n$ given by $x^i(v):=u^i\circ \pi (v)=u^i(p)=:p^i$, where  $\pi \colon TM \to M$ is the canonical projection (the "foot map") and $y^i(v)=v(u^i)$, $i \in \{1, \ldots, n\}$.  
\begin{Def} A Finsler function is a continuous function $F\colon TM\to \mathbb{R}$ satisfying the following conditions:  $\displaystyle{F}$ is smooth on the complement of the zero section \emph{(}regularity\emph{)}, $\displaystyle{F(tv)=tF(v)}$ for all $\displaystyle{t> 0}$ \emph{(}positive homogeneity\emph{)} and the Hessian 
\[ (g_{ij})=\left(\frac{\partial^2 E}{\partial y^i \partial y^j}\right) \]
of the energy function $E=F^2/2$ is positive definite at all nonzero elements $\displaystyle{v\in T_pM}$ \emph{(}strong convexity\emph{)}. The pair $(M,F)$ is called a Finsler manifold.
\end{Def}

\begin{Def} Let $(M,F)$ be a Finsler manifold. A linear connection $\nabla$ on $M$ is called \emph{compatible} to $F$ if the parallel transports with respect to $\nabla$ preserve the Finslerian length of tangent vectors. Then the triplet $(M,F,\nabla)$ is called a generalized Berwald manifold.
\end{Def}

It is well-known that any linear connection $\nabla$ determines a horizontal distribution spanned by the \emph{horizontal vector fields}
 \[ X_i^{h}:=\frac{\partial}{\partial x^i}-y^j {\Gamma}^k_{ij}\circ \pi \frac{\partial}{\partial y^k}, \quad i \in \{1, \ldots, n\}. \]
The vanishing of the derivative of $F$ along the integral curves of horizontal vector fields characterizes the compatibility of $\nabla$ to $F$ \cite{V14}. Therefore a linear connection $\nabla$ is compatible to a Finsler function if and only if  
\begin{equation}
\label{compeqnatur}
\frac{\partial F}{\partial x^i}-y^j {\Gamma}^k_{ij}\circ \pi \frac{\partial F}{\partial y^k}=0, \quad i \in \{1, \ldots, n\}.
\end{equation}
System (\ref{compeqnatur}) is called \emph{compatibility equations}.

\begin{Def} A Riemannian metric $\gamma$ on a Finsler manifold is \emph{compatible} to $F$ if every linear connection compatible to $F$ is metrical with respect to $\gamma$. 
\end{Def}

It is well-known that the so-called \emph{averaged Riemannian metric} (defined by the integration of the metric components $g_{ij}$ over the indicatrices) is compatible to the Finsler function, for details see \cite{V5} and \cite{V11}. Compatible Riemannian metrics can be introduced in many different ways as well \cite{Cram}, \cite{Mat1}. They can be directly given as the Riemannian part of the Finsler function in case of a Randers manifold \cite{Vin1}. In what follows, let $\gamma$ be a given compatible Riemannian metric to $F$ with its uniquely determined L\'{e}vi-Civita connection $\nabla^*$.
The horizontal vector fields generated by $\nabla^*$ are
\[ X_i^{h^*}:=\frac{\partial}{\partial x^i}-{y^j \Gamma}^{*k}_{ij}\circ \pi \frac{\partial}{\partial y^k},  \quad  i \in \{1, \ldots, n\}. \]

\section{The compatibility equations in terms of the torsion components}

Following the idea in \cite{V14}, we formulate the compatibility equations \eqref{compeqnatur} in terms of the torsion components of the compatible linear connection $\nabla$ instead of the connection parameters. Using the Christoffel process, we have the formula
\begin{equation}
\label{torsion}
\Gamma_{ij}^r=\Gamma_{ij}^{*r}-\frac{1}{2}\left(T^{l}_{jk}\gamma^{kr}\gamma_{il}+T^{l}_{ik}\gamma^{kr}\gamma_{jl}-T_{ij}^r\right),
\end{equation}
and system \eqref{compeqnatur} can be written as 
\[ X_i^{h^*}F+\frac{1}{2}y^j \left(T^{l}_{jk}\gamma^{kr}\gamma_{il}+T^{l}_{ik}\gamma^{kr}\gamma_{jl}-T_{ij}^r\right)\circ \pi \frac{\partial F}{\partial y^r}=0, \quad  i \in \{1, \ldots, n\}. \]
Evaluating at a nonzero tangent vector $v \in T_p M$, we have the inhomogeneous system  
\[ y^j (v)\left(T^{l}_{jk}\gamma^{kr}\gamma_{il}+T^{l}_{ik}\gamma^{kr}\gamma_{jl}-T_{ij}^r\right)(p) \frac{\partial F}{\partial y^r}(v)=-2X_i^{h^*}F(v), \quad  i \in \{1, \ldots, n\} \]
of linear equations for the components $T_{ab}^{c}(p)$. The set $A_p(v)$ of its solutions is an affine subspace of the vector space $\wedge^2 T_p^*M \otimes T_pM$ spanned by 
\[ du^i_p \wedge du^j_p \otimes \left( \frac{\partial}{\partial u^k}\right)(p), \quad 1\leq i < j\leq n, \quad  k\in \{1, \ldots, n\}. \] 
It is of dimension $\displaystyle{\binom{n}{2}n}$. Running through all nonzero tangent vectors $v \in T_p M$, the solution set 
$$A_p=\bigcap_{v\in T_pM\setminus \{ \bf{0}\}} A_p(v)$$
containing the restrictions of the torsion tensors of all compatible linear connections to the Cartesian product $\displaystyle{T_pM\times T_pM}$ is also an affine subspace of $\displaystyle{\wedge^2 T_p^*M \otimes T_pM}$ (as an intersection of affine subspaces).

\begin{Thm}
\label{smoothness} If $(M, F, \nabla)$ is a connected generalized Berwald manifold, then the mapping $p\in M\mapsto A_p\subset \wedge^2 T_p^*M \otimes T_pM$ is a smooth affine distribution of constant rank on the torsion tensor bundle.
\end{Thm}

\begin{proof} Let  $\displaystyle{H_p}\subset \wedge^2 T_p^*M \otimes T_pM$ be the directional space of the affine subspace $A_p$, and let $\nabla$ be a compatible linear connection. We are going to prove that  
\begin{equation}
\label{action}
T_q (v,w):=\varphi\circ  T_p (\varphi^{-1}(v), \varphi^{-1}(w))
\end{equation}
belongs to the directional space $H_q$ for any $T_p\in H_p$, where $\varphi$ is a linear isometry between the tangent spaces at the corresponding points $p$ and $q$. Such an isometry is generated by parallel transports with respect to $\nabla$. Let $(Q_i^j)$ be the matrix representation of $\varphi$, $(P_j^i)=\left(Q_i^j\right)^{-1}$. The directional space $H_q \subset \wedge^2 T_q^*M \otimes T_qM$ can be given as 
\begin{equation}
\label{condextc} \frac{1}{2}y^j\circ \varphi (v) \left(T^{l}_{jk}\gamma^{kr}\gamma_{il}+T^{l}_{ik}\gamma^{kr}\gamma_{jl}-T_{ij}^r\right)(q) \frac{\partial F}{\partial y^r}\circ \varphi (v)=0,  \quad  i \in \{1, \ldots, n\},
\end{equation}
where $v$ runs through the nonzero elements of $T_pM$. So does $\varphi(v)$ in $T_q M$. Therefore (\ref{condextc}) is equivalent to
$$ \frac{1}{2}y^b (v)\left(T^{l}_{jk}Q_b^j P_r^c \gamma^{kr}\gamma_{il}+T^{l}_{ik}Q_b^j P_r^c\gamma^{kr}\gamma_{jl}-T_{ij}^rQ_b^j P_r^c\right)(q) \frac{\partial F}{\partial y^c}(v)=0,$$
where  $ i \in \{1, \ldots, n\}$ and $v\in T_pM$. Indeed, $\displaystyle{\varphi^j=y^j\circ \varphi=Q_b^j y^b}$ and the invariance property $F\circ \varphi=F$ imply that 
$$\frac{\partial F}{\partial y^c}(v)=\frac{\partial F\circ \varphi}{\partial y^c}(v)=Q_c^r\frac{\partial F}{\partial y^r}\circ \varphi(v) \ \Rightarrow \ P_r^c\frac{\partial F}{\partial y^c}(v)=\frac{\partial F}{\partial y^r}\circ \varphi(v).$$
Using the equalities
$$Q_a^i Q_b^j\gamma_{ij}(q)=\gamma_{ab}(p), \ P_i^a P_j^b \gamma^{ij}(q)=\gamma^{ab}(p), \ Q_b^j \gamma_{jl}(q)=P_l^j \gamma_{jb}(p) \ \ \textrm{and}\ \ P_j^b \gamma^{jk}(q)=Q_j^k\gamma^{jb}(p),$$
we have  
$$\frac{1}{2}y^b (v)\left(T^{l}_{jk}(q)Q_b^j Q_r^k \gamma^{rc}(p)\gamma_{il}(q)+T^{l}_{ik}(q)Q_b^j Q_r^k\gamma^{rc}(p)\gamma_{jl}(q)-T_{ij}^r(q) Q_b^j P_r^c\right) \frac{\partial F}{\partial y^c}(v)=0.$$
Taking the product by the matrix $(Q^i_a)$, an equivalent system of equations is
$$\frac{1}{2}y^b (v)\left(T^{l}_{jk}(q)Q_a^i Q_b^j Q_r^k \gamma^{rc}(p)\gamma_{il}(q)+T^{l}_{ik}(q)Q_a^i Q_b^j Q_r^k\gamma^{rc}(p)\gamma_{jl}(q)-T_{ij}^r (q)Q_a^i Q_b^j P_r^c\right) \frac{\partial F}{\partial y^c}(v)=0,$$
$$\frac{1}{2}y^b (v)\left(T^{l}_{jk}(q)P_l^i Q_b^j Q_r^k \gamma^{rc}(p)\gamma_{ia}(p)+T^{l}_{ik}(q)Q_a^i P_l^j Q_r^k\gamma^{rc}(p)\gamma_{jb}(p)-T_{ij}^r (q)Q_a^i Q_b^j P_r^c\right) \frac{\partial F}{\partial y^c}(v)=0,$$
$$\frac{1}{2}y^b (v)\left(T_{br}^i \gamma^{rc}\gamma_{ia}+ T_{ar}^j \gamma^{rc}\gamma_{jb}- T_{ab}^c\right)(p) \frac{\partial F}{\partial y^c}(v)=0, \quad  a \in \{1, \ldots, n\}$$ 
for any nonzero $v\in T_pM$, where $T_{ab}^c(p):=T_{ij}^r (q)Q_a^i Q_b^j P_r^c$, i.e.,  $T_{ij}^r(q)=P^a_i P_j^b T_{ab}^c(p) Q_c^r$ as in formula (\ref{action}).
\end{proof}

Now we rewrite the compatibility equations in a more compact form. Due to the skew-symmetry, it is enough to keep the components $T_{ab}^{c}$ with indices $a<b$. Denoting by $\sigma_{ab;i}^{c}$ the coefficient of $T_{ab}^{c}$ ($a<b$) in the $i$-th equation, we have 
\[ \sigma_{ab;i}^{c}=\left(y^a \gamma^{br}-y^b\gamma^{ar}\right)\frac{\partial F}{\partial y^r}\gamma_{ic}+\left(\delta_i^a \gamma^{br}-\delta_i^b\gamma^{ar}\right)\frac{\partial F}{\partial y^r}y^j\gamma_{jc}-\left(\delta_i^a y^b-\delta_i^b y^a\right)\frac{\partial F}{\partial y^c}. \]
If $\displaystyle{(\partial/\partial u^1, \ldots, \partial/\partial u^n)}$ is an orthonormal basis at $p\in M$ with respect to the compatible Riemannian metric $\gamma$, then
\begin{equation}
\sigma_{ab;i}^{c}=\delta_i^c\left(y^a \frac{\partial F}{\partial y^b}-y^b\frac{\partial F}{\partial y^a}\right)+\delta_i^a \left(y^c\frac{\partial F}{\partial y^b}-y^b\frac{\partial F}{\partial y^c}\right)-\delta_i^b \left(y^c\frac{\partial F}{\partial y^a}-y^a\frac{\partial F}{\partial y^c}\right).
\end{equation}
Introducing the notation
\begin{equation}
\label{fab}
f_{ab}=y^a \frac{\partial F}{\partial y^b}-y^b\frac{\partial F}{\partial y^a},
\end{equation}
the coefficients are
\begin{equation}
\label{coeffort}
\sigma_{ab;i}^{c}=\delta_i^c f_{ab}+\delta_i^a f_{cb}+\delta_i^b f_{ac},
\end{equation}
and the compatibility equations take the form
\begin{equation}
\label{compeqinnprod}
 \sum_{a<b,c} \sigma_{ab;i}^{c} T_{ab}^{c} = -2X_i^{h^*}F, \quad  i \in \{1, \ldots, n\},
\end{equation}
where the summation symbol means summing over the indices
\[ \{ (a,b,c) \in \nn_n \times \nn_n \times \nn_n \, | \, a < b \}, \quad \nn_n:=\{1, \dots, n\}. \]

\subsection{The case of Finsler surfaces} In case of dimension two, the compatibility equations are easy to solve because we have only two variables $T_{12}^1$ and $T_{12}^2$ with coefficients $\sigma_{12; i}^1=2\delta_i^1 f_{12}$ and $\sigma_{12; i}^2=2\delta_{i}^2f_{12}$. The compatibility equations are 
$$f_{12} T_{12}^{1} = -X_1^{h^*}F, \quad f_{12} T_{12}^{2} = -X_2^{h^*}F.$$
The uniquely determined solutions can be expressed (for example) in terms of quantities given by integration along the Euclidean unit circle with respect to the compatible Riemannian metric in the tangent spaces:
$$T_{12}^{1} = -\int_{c_*} f_{12} X_1^{h^*}F \bigg/\int_{c_*} f^2_{12}, \quad T_{12}^{2} = -\int_{c_*} f_{12} X_2^{h^*}F\bigg/\int_{c_*} f^2_{12}.$$
If $f_{12}$ is identically zero in $T_pM$, then the Finslerian unit circle is homothetic to the Riemannian unit circle because $F\circ c_*$ is constant. In case of a generalized Berwald surface, this means that it is a Riemannian surface due to the invariance of the indicatrices under the parallel transports given by a linear connection. Another aspects and examples in case of dimension 2 can be found in \cite{VOM}, \cite{KVO} and \cite{KVMO}.
 
\subsection{The extremal compatible linear connection} Consider a generalized Berwald space. As we have seen above, the pointwise solution sets $A_p$ of the compatibility equations are affine subspaces of the vector spaces $\displaystyle{\wedge^2 T_p^*M \otimes T_pM}$ ($p \in M$). This means that the global solution is not necessarily unique even we have a global solution. The unicity can be provided by some additional requirements. Following the idea in \cite{V14} we are looking for the so-called extremal compatible linear connection minimizing the norm of its torsion at each point of the manifold in the following sense.

\begin{Def}
\label{indmetric} Let $(M,F)$ be a Finsler manifold with a compatible Riemannian metric $\gamma$, and suppose that the coordinate vector fields $\partial/\partial u^1, \ldots, \partial/\partial u^n$ form an orthonormal basis at a point $p\in M$ with respect to $\gamma$. We introduce a Riemannian metric on $\wedge^2 T_p^*M \otimes T_pM$ in the following way:
if $\displaystyle{T=\sum_{i<j,k} T_{ij}^k du^i \wedge du^j \otimes \frac{\partial}{\partial u^k}}$, then
\begin{equation}
\label{torsionnorm}
\dotprod{T_p, S_p}: =\sum_{i<j,k} T_{ij}^k(p)S_{ij}^k(p).
\end{equation}
The extremal compatible linear connection on a generalized Berwald space $M$ is the uniquely determined compatible linear connection whose torsion minimizes the norm arising from \emph{(}\ref{torsionnorm}\emph{)}.
\end{Def}

Since the solution set $A_p\subset \wedge^2 T_p^*M \otimes T_pM$ of the compatibility equations at $p\in M$ is an affine (especially, convex) set, its closest element to the origin is uniquely determined. In case of a generalized Berwald manifold the pointwise extremal solutions form a continuous section of the torsion bundle. Therefore the connection parameters of the corresponding linear connection are obviously continuous. Conversely, the existence of a continuous compatible linear connection implies that the manifold is monochromatic and, consequently, it is a generalized Berwald manifold; for details see \cite{BM} and \cite{V14}.

\section{The compatibility equations in 3D}

In what follows, we are going to solve the compatibility equations for 3-dimensional Finsler manifolds. All compatible linear connections will be determined by their torsion components and we will find the extremal one among them. The trick is to group the variables and work in a 3D space three times instead of the $9$-dimensional fiber of the torsion tensor bundle.

\subsection{Geometric structures on the tangent spaces} Let $p$ be a given point of $M$. 
\begin{enumerate}
\item The Finsler function $F$ restricted to $T_p M$ is a Minkowski norm, i.e., the level surface $F^{-1}(1)\cap T_p M$ is the boundary of a  (strictly) convex body containing the origin in its interior. It is the Finslerian indicatrix or, unit sphere, at the point $p\in M$.
\item The compatible Riemannian metric $\gamma$ restricted to $T_p M$ is the standard Euclidean metric after choosing the orthonormal basis $\displaystyle{\partial/\partial u^1, \partial/\partial u^2,  \partial/\partial u^3}$ at the point $p\in M$. Its indicatrix is the Euclidean unit sphere. 
\end{enumerate}

In what follows we simply write $\dotprod{v,w}$ instead of $\gamma_p(v,w)$. Orthogonality, norms and normal vectors are taken in the usual Euclidean sense. Since the Finslerian spheres are the level sets of $F$, the Euclidean gradient vector field 
\[ G:=\mathrm{grad} \ F=  \dfrac{\partial F}{\partial y^1}\dfrac{\partial}{\partial y^1}+\dfrac{\partial F}{\partial y^2}\dfrac{\partial}{\partial y^2}+\dfrac{\partial F}{\partial y^3}\dfrac{\partial}{\partial y^3} \sim \left( \dfrac{\partial F}{\partial y^1},\dfrac{\partial F}{\partial y^2},\dfrac{\partial F}{\partial y^3} \right) = \left(\partial_1 F, \partial_2 F, \partial_3 F \right)  \]
gives the (outer) normals. The (outer) normal vectors of the tangent planes of a Euclidean sphere are given by the radial or Liouville vector field
\[ C:= y^1 \dfrac{\partial}{\partial y^1}+y^2\dfrac{\partial}{\partial y^2}+y^3\dfrac{\partial}{\partial y^3} \sim \left( y^1,y^2,y^3 \right). \]
In 3D two tangent planes are either parallel or intersect in a common line. 

\begin{Def} {\emph{\cite{V14}}} We call a nonzero vector $v \in T_p M$ 
\begin{itemize}
\item a \emph{vertical contact point}, if the Finslerian and Riemannian spheres have the same tangent planes at $v$, i.e., the Euclidean normal vectors $G_v$ and $C_v$ are parallel;
\item a \emph{horizontal contact point}, if $X_i^{h^*}F(v)=0$ for all $i\in \{1, \ldots, n\}$.
\end{itemize}
The tangent space $T_p M$ is called vertical/horizontal contact if all of its nonzero elements are vertical/horizontal contact.
\end{Def}

We briefly summarize some basic facts about the vertical and horizontal contact points, for details we refer to \cite{V14}.

\begin{itemize}
\item At vertical contact points the coefficients $\sigma_{ab;i}^{c}$ in the compatibility equations are zero. 
\item If the compatibility equations are solvable, i.e., we have a generalized Berwald manifold, then all vertical contact points are horizontal contact. 
\item If we have a \emph{connected} generalized Berwald manifold with a vertical contact tangent space, then it is a Riemannian manifold and its extremal compatible linear connection is the L\'{e}vi--Civita connection of the compatible Riemannian metric. 
\item If $T_pM$ is a horizontal contact tangent space, then the torsion of the extremal compatible linear connection at $p$ is zero because the system of the compatibility equations is homogeneous.
\end{itemize}

Suppose that the tangent space $T_pM$ contains at least one element $v$ that is not vertical contact. Since $G_v$ and $C_v$ are linearly independent 
\begin{itemize}
\item we can use the Gram--Schmidt process to construct an orthogonal pair by substituting $G$ with
$$G^{\bot}_v:=G_v-\frac{\dotprod{G_v,C_v}}{\dotprod{C_v, C_v}}C_v=G_v-\frac{C_v F}{\dotprod{v,v}}C_v=G_v-\frac{F(v)}{\dotprod{v,v}}C_v$$
because of the first order homogeneity of the Finsler function. 
\item If $C \times G$ is the cross product of the vector fields $C$ and $G$, then
\end{itemize}
\begin{equation}
\label{n}
C\times G=\begin{vmatrix}
e_1 & e_2 & e_3 \\
y^1 & y^2 & y^3 \\
\partial_1 F & \partial_2 F & \partial_3 F
\end{vmatrix} = \colvec{y^2 \partial_3  F - y^3 \partial_2  F \\ y^3 \partial_1  F - y^1 \partial_3  F \\ y^1 \partial_2  F - y^2 \partial_1  F } = \colvec{f_{23}\\f_{31}\\f_{12}}.
\end{equation}

It is clear that the cross product vanishes at the vertical contact elements in the tangent spaces and vice versa. According to the orthogonality to both $C_v$ and $G_v$, it lies in the tangent planes of both the Euclidean and the Finslerian spheres passing through $v$. Finally, $(C,G^{\bot},C\times G)$ is an orthogonal frame on the complement of the vertical contact elements.  

\begin{Rem} {\emph{For any $p\in M$, the tangent space $T_p M$ has a vertical contact point. Otherwise the vector field $C\times G$ would be a (continuous) non-vanishing tangent vector field to the 3-dimensional Euclidean unit sphere contradicting the hedgehog theorem. Another possible argument independently of the dimension is based on the furthest point property and convexity: the tangent plane to the Finslerian unit sphere at the furthest point $v$ from the origin with respect to the Euclidean metric must be orthogonal to $v$.}}
\end{Rem}

\begin{Thm} In case of a three-dimensional Finsler manifold, the compatibility equations can be written in the form
\begin{equation}
\label{comp1}
\left<C\times G,\begin{pmatrix}\  -T_{12}^3+T_{13}^2+T_{23}^1\\-2T_{13}^1\\ \ \ 2T_{12}^1 \end{pmatrix}\right> = -2X_1^{h^*}F,
\end{equation}
\begin{equation}
\label{comp2}
\left<C\times G,\begin{pmatrix} 2T_{23}^2\\-T_{12}^3-T_{13}^2-T_{23}^1\\2T_{12}^2 \end{pmatrix}\right> = -2X_2^{h^*}F,
\end{equation}
\begin{equation}
\label{comp3}
\left<C\times G,\begin{pmatrix} \ \ 2T_{23}^3\\-2T_{13}^3 \\T_{12}^3+T_{13}^2-T_{23}^1 \ \ \end{pmatrix}\right> = -2X_3^{h^*}F,
\end{equation}
where the new variables are groups of the torsion components and
$$G:=\mathrm{grad}\  F \sim \left( \dfrac{\partial F}{\partial y^1},\dfrac{\partial F}{\partial y^2},\dfrac{\partial F}{\partial y^3} \right) = \left(\partial_1 F, \partial_2 F, \partial_3 F \right). $$
\end{Thm}

\begin{proof}
Recall that the coefficient $\sigma_{ab;i}^{c}$ of the torsion component $T_{ab}^c$ in the $i$-th compatibility equation is
\begin{equation}
\sigma_{ab;i}^{c}=\delta_i^c f_{ab}+\delta_i^a f_{cb}+\delta_i^b f_{ac}, \ \ \textrm{where}\ \ f_{ab}=y^a \frac{\partial F}{\partial y^b}-y^b\frac{\partial F}{\partial y^a}, \  \ \textrm{and so on}.
\end{equation}

Depending on the repetition among the indices $a<b$ and $c$, we have $8$ cases as the following table shows. In the first row, they can take any value different from $ i$. Otherwise, equal indices are in the same cells, and different cells contain different indices. 

\vspace{0.3cm}
\begin{center}
\begin{tabular}{|c||c|c|c|c||l|}
\hline 
 & \multicolumn{4}{c||}{\textrm{indices}} & \textrm{the coefficients} \\ 
\hline 
\hline
1 & $i$ & $a$ & $b$ & $c$ & $\sigma_{ab;i}^{c}=0$\\ 
\hline 
2 & $i=a$ & $b$ & $c$ &  & $\sigma_{ib;i}^{c}=f_{cb}$ \\ 
\hline 
3 & $i=a$ & $b=c$ &   &   & $\sigma_{ib;i}^{b}=0$ \\ 
\hline 
4 & $i=b$ & $a$ & $c$ &   & $\sigma_{ai;i}^{c}=f_{ac}$ \\ 
\hline 
5 & $i=b$ & $a=c$ &   &   & $\sigma_{ai;i}^{a}=0$ \\ 
\hline 
6 & $i=c$ & $a$ & $b$ &   & $\sigma_{ab;i}^{i}=f_{ab}$ \\ 
\hline 
7 & $i=a=c$ & $b$ &   &   & $\sigma_{ib;i}^{i}=2f_{ib}$ \\ 
\hline 
8 & $i=b=c$ & $a$ &   &   & $\sigma_{ai;i}^{i}=2f_{ai}$ \\ 
\hline 
\end{tabular}
\end{center}
\vspace{0.3cm}

To sum up, the surviving coefficients in the $i$-th equation contain at least one index equal to $i$ and, if exactly one index equals to $i$, then the remaining two indices must be different. Thus, the compatibility equations in 3D (in matrix form) are

\vspace{0.3cm}
\begin{center}
\begin{tabular}{|c||c|c|c|c|c|c|c|c|c||c|}
\hline
 & $T_{12}^{1}$ & $T_{12}^{2}$ & $T_{12}^{3}$ & $T_{13}^{1}$ & $T_{13}^{2}$ & $T_{13}^{3}$ & $T_{23}^{1}$ & $T_{23}^{2}$ & $T_{23}^{3}$ & RHS \\ 
\hline 
\hline 
1 & $2f_{12}$ & $0$ & $f_{32}$ & $2f_{13}$ & $f_{23}$ & $0$ & $f_{23}$ & $0$ & $0$ & $-2X_1^{h^*}F$ \\ 
\hline 
2 & $0$ & $2f_{12}$ & $f_{13}$ & $0$ & $f_{13}$ & $0$ & $f_{13}$ & $2f_{23}$ & $0$ & $-2X_2^{h^*}F$ \\ 
\hline 
3 & $0$ & $0$ & $f_{12}$ & $0$ & $f_{12}$ & $2f_{13}$ & $f_{21}$ & $0$ & $2f_{23}$ & $-2X_3^{h^*}F$ \\ 
\hline 
\end{tabular},
\vspace{0.3cm}
\end{center}
and we can group the components in the following way:
\vspace{0.3cm}
\begin{center}
\begin{tabular}{ccccccc}
$f_{23}(-T_{12}^3+T_{13}^2+T_{23}^1)$ & $-$ & $2f_{31}T_{13}^1$ & $+$ & $2f_{12}T_{12}^1$ & $=$ & $-2X_1^{h^*}F$, \\ 
$2f_{23}T_{23}^2$ & $+$ & $f_{31}(-T_{12}^3-T_{13}^2-T_{23}^1)$ & $+$ & $2f_{12}T_{12}^2$ & $=$ & $-2X_2^{h^*}F$, \\ 
$2f_{23}T_{23}^3$ & $-$ & $2f_{31}T_{13}^3$ & $+$ & $f_{12}(T_{12}^3+T_{13}^2-T_{23}^1)$ & $=$ & $-2X_3^{h^*}F$. 
\end{tabular}
\vspace{0.3cm}
\end{center}

Since the coefficients are those of the cross product defined by formula \eqref{n}, we can indeed write the compatibility equations in the form (\ref{comp1}), (\ref{comp2}) and (\ref{comp3}).
\end{proof}

\begin{Lem} 
\label{lemma:03} For any not vertical contact tangent space $T_pM$, the common directional space of $(\ref{comp1})$, $(\ref{comp2})$ and $(\ref{comp3})$ is trivial or it is a one-dimensional linear subspace in $T_pM$, which is the rotational axis of the Finslerian indicatrix. 
\end{Lem}

\begin{proof} The common directional space is given by 
\begin{equation}
\label{directional}
\left<C\times G,\begin{pmatrix} t^1\\t^2 \\ t^3\end{pmatrix}\right> = 0.
\end{equation}
Taking a not vertical contact element $v\in T_pM$, (\ref{directional}) can be refined as follows: since $C_v\times G_v\neq {\bf{0}}$ we can consider the non-trivial integral curve $c$ of the cross product passing through $v$. Since $c'=(C\times G)\circ c$, it follows that $c'(t) \ \bot \ c(t)\sim C_{c(t)}$ with respect to the Euclidean inner product, i.e., 
$$\dotprod{c,c}=\textrm{const.} \ \ \Rightarrow \ \ \dotprod {c', c}(0)=0 \ \ \Rightarrow \ \ \dotprod {c'', c}(0)=-\dotprod{c', c'}(0)=-\dotprod{C_v\times G_v, C_v\times G_v}\neq 0.$$
This means that $c'(0)$ and $c''(0)$ are linearly independent. Therefore, evaluating (\ref{directional}) along $c$, a simple differentiation gives 
\begin{equation}
\label{solutionlines}
\left<c'(0),\begin{pmatrix} t^1\\t^2\\t^3 \end{pmatrix}\right> =0, \ \ \left<c''(0),\begin{pmatrix} t^1\\t^2\\t^3 \end{pmatrix}\right> =0.
\end{equation}
Hence the directional space is at most one-dimensional. Suppose that $\vec{t}\neq {\bf 0}$ is a solution of (\ref{directional}) for any not vertical contact $v\in T_pM$. In case of vertical contact elements, (\ref{directional}) is automatically satisfied because of the vanishing of the vector field $C\times G$. It follows that 
$$0=\left<C_v\times G_v,\vec{t}\ \right> = -\left<C_v\times \vec{t}, G_v\right>= C_v\times \vec{t} \ F \quad (v\in T_pM).$$
Since the integral curves of $C\times \vec{t}$ must satisfy the equation $c(t)\times \vec{t}=c'(t)$, we have $c(t)=e^{tA}c(0)$, where $A$ is the (skew-symmetric) matrix of the linear transformation $v\mapsto v\times \vec{t}$. This means that $(F\circ c)'(t)=0$, i.e., $F$ is constant along the orbits under the action of the one-parameter rotational group generated by $A$. 
\end{proof}

\begin{Cor}  If $(M, F, \nabla)$ is a connected generalized Berwald manifold of dimension three, then the mapping $p\in M\mapsto A_p\subset \wedge^2 T_p^*M \otimes T_pM$ is a smooth affine distribution on the torsion tensor bundle of constant rank $0$ or $1$.
\end{Cor}

\begin{Thm}
\label{thm:02} If $M$ is a connected non-Riemannian generalized Berwald manifold of dimension three, then we have the following possible cases.
\begin{itemize}
\item[(UDC)]
The Finslerian indicatrix is a Euclidean surface of revolution at some and therefore all points of the manifold and the rotational axes are generated by a globally well-defined nowhere vanishing covariant constant vector field $D\in \mathfrak{X}(M)$ with respect to any compatible linear connection. One of them is given by
\begin{equation}
\label{oneofthem}
\nabla_X Y=\nabla^*_X Y+\frac{\dotprod{\nabla^*_X D, Y}D-\dotprod{Y,D}\nabla^*_X D}{K^2},
\end{equation}
where $K^2$ is the constant norm square of $D$.
\item[(DC)] We have a uniquely determined flat compatible linear connection given by 
\begin{equation}
\label{flat}
\nabla_X Y=\nabla^*_X Y -\rho(X)\times Y,
\end{equation} 
where $\rho$ is an endomorphism of  $\mathfrak{X}(M)$ satisfying
$$R^*(X,Y)Z=\bigg((\nabla^*_X \rho)(Y)-(\nabla^*_Y \rho)(X)-\rho(X)\times \rho(Y)\bigg)\times Z.$$
\end{itemize}
\end{Thm}

\begin{proof} In the undetermined case (UDC) let $p$ be a point such that (\ref{directional}) has a non-zero solution $\vec{t}$ in $T_pM$. By the previous lemma, the Finslerian indicatrix at the point $p$ is a Euclidean surface of revolution. Using parallel transports with respect to one of the compatible linear connections from $p$ to an arbitrary point of the manifold, it follows that all Finslerian indicatrices are Euclidean surfaces of revolution. Since we have a non-Riemannian generalized Berwald manifold, the rotational axes must be uniquely determined  \cite{V12} (Corollary 8). Choosing a directional vector at a single point $p\in M$, we can extend it to a globally well-defined nowhere vanishing covariant constant vector field $D\in  \mathfrak{X}(M)$ by parallel transports with respect to a compatible linear connection. Especially, $\nabla D=0$, i.e., the parallel transports with respect to $\nabla$ preserve $D$. Conversely, since the Finslerian indicatrices are surfaces of revolution and they are isometric to each other, it is enough to keep the rotational axes invariant by the parallel transports of a metric linear connection in the Riemannian sense to have a compatible linear connection to the Finsler function. Therefore  (\ref{oneofthem}) is one of them. If we have (DC) (determined case), then the solution space of  (\ref{directional}) is trivial at any point of the manifold. To avoid Euclidean surfaces of revolution and (UDC), the unit component of the holonomy group of the compatible linear connection must be trivial. Therefore the compatible linear connection is flat because $R_p(X,Y)\neq {\bf 0}$ would generate a one-parameter subgroup in the holonomy group leaving the indicatrix invariant. Since
$$\nabla^*_X Y=\nabla_X Y+A(X,Y)$$
for some tensor field of type ($1, 2$), it follows that $\nabla$ is a metric linear connection if and only if
$$\dotprod{A(X,Y), Z}=-\dotprod{A(X,Z), Y}.$$
The skew-symmetry implies that  
$$\nabla^*_X Y=\nabla_X Y+\rho(X)\times Y.$$
A simple calculation shows
$$\nabla_X \nabla_Y Z=\nabla^*_X \nabla_Y Z-\rho(X)\times \nabla_Y Z=$$
$$\nabla^*_X \bigg(\nabla^*_Y Z-\rho(Y)\times Z\bigg)-\rho(X)\times \bigg(\nabla^*_Y Z-\rho(Y)\times Z\bigg)=$$
$$\nabla^*_X \nabla^*_Y Z-\bigg((\nabla^*_X \rho) (Y)-\rho(\nabla^*_X Y)\bigg)\times Z-\rho(Y)\times \nabla^*_X Z-\rho(X)\times \nabla^*_Y Z+$$
$$\rho(X)\times \bigg(\rho(Y)\times Z\bigg).$$
For the sake of simplicity suppose that $[X,Y]=0$ before changing the role of $X$ and $Y$. We have that
$$R(X,Y)Z=R^*(X,Y)Z-(\nabla^*_X \rho) (Y)\times Z+ (\nabla^*_Y \rho) (X)\times Z +\bigg(\rho(X)\times \rho(Y)\bigg)\times Z$$
because of $\nabla^*_X Y-\nabla^*_Y X=0$ and, using the Jacobi identity for the cross product. Since $\nabla$ is flat, the formula for the curvature of the L\'{e}vi-Civita connection  follows immediately. 
\end{proof}

\subsection{General method for solution} Following the technique in the proof of Lemma \ref{lemma:03}, let us choose a not vertical contact elemet $v\in T_pM$. The solution of (\ref{solutionlines}) gives a one-dimensional directional space. Using the right hand sides of  (\ref{comp1}), (\ref{comp2}) and (\ref{comp3}), respectively, the solutions of the inhomogeneous version of (\ref{solutionlines}) give one-dimensional affine subspaces.  Therefore we can speak about  \emph{solution lines} belonging to a not vertical contact element $v\in T_pM$.  The vertical contact elements must be horizontal contact.  The possible cases are the following:
\begin{itemize}
\item there is no solution of the compatibility equations,
\item (determined case) the solution is uniquely determined, i.e., the corresponding solution lines intersect each other at the same point for all not vertical contact elements and all vertical contact elements are horizontal contact;
\item (undetermined case) all vertical contact elements are horizontal contact and the corresponding solution lines coincide for all not vertical contact element in $T_pM$. The common directional space is generated by the rotational axis of the Finslerian indicatrix. 
\end{itemize}
Consider the equations  
\begin{equation}
\label{compgeneral}
\left<C\times G,\begin{pmatrix} t_i^1\\t_i^2 \\ t_i^3\end{pmatrix}\right> = b_i, \quad   i\in \{1, 2, 3\}
\end{equation}
in the tangent spaces. Especially,
$$b_i=-2X_i^{h^*}F;$$
cf. equations (\ref{comp1}), (\ref{comp2}) and (\ref{comp3}). 

\begin{Lem} Let $G=\mathrm{grad} \ F$ be the Euclidean gradient on the Finslerian spheres with respect to the compatible Riemannian metric, and consider the vector field $C\times G$. Taking the integral curve $c$ starting from a not vertical contact element $v$ in the tangent space $T_pM$, the solutions of \eqref{compgeneral} at the point $p$ are of the form
\begin{equation}
\label{compgeneralsol}
\begin{pmatrix} \vec{t}_1\\ \vec{t}_2\\ \vec{t}_3\end{pmatrix}=\begin{pmatrix} \omega_{11} & \omega_{12}  & \omega_{13} \\ \omega_{21} & \omega_{22}  & \omega_{23} \\ \omega_{31} & \omega_{32}  & \omega_{33} \end{pmatrix}\begin{pmatrix} c'(0)\\ c''(0)\\ c'(0)\times c'' (0)\end{pmatrix},
\end{equation}
where
$$\omega_{i1}=\frac{b_i(v)}{\ |c'(0)|^2}-\frac{\langle c'(0),c''(0)\rangle}{|c'(0)|^5 \kappa^2(0)} \left(\frac{b_i\circ c}{|c'|}\right)'(0), \ \omega_{i2}=\frac{1}{|c'(0)|^3 \kappa^2(0)} \left(\frac{b_i\circ c}{|c'|}\right)'(0)$$
and $b_i=-2X_i^{h^*}F$,  $i\in \{1, 2, 3\}$.
\end{Lem}

\begin{proof}
Taking a not vertical contact element $v\in T_pM$, (\ref{compgeneral}) can be refined as follows: since $C_v\times G_v\neq {\bf{0}}$ we can consider the non-trivial integral curve $c$ of the cross product starting at $v$. Since $c'=(C\times G)\circ c$, it follows that $c'(t) \ \bot \ c(t)\sim C_{c(t)}$ with respect to the Euclidean inner product, i.e., 
$$\dotprod{c,c}=\textrm{const.} \ \ \Rightarrow \ \ \dotprod {c', c}(0)=0 \ \ \Rightarrow \ \ \dotprod {c'', c}(0)=-\dotprod{c', c'}(0)=-\dotprod{C_v\times G_v, C_v\times G_v}\neq 0.$$
This means that $c'(0)$ and $c''(0)$ are linearly independent. Therefore, evaluating (\ref{compgeneral}) along $c$, a simple differentiation gives 
$$\left<c'(0),\begin{pmatrix} t_i^1\\t_i^2\\t_i^3 \end{pmatrix}\right> =b_i (v), \ \left<c''(0),\begin{pmatrix} t_i^1\\t_i^2\\t_i^3 \end{pmatrix}\right> =(b_i\circ c)'(0), \quad  i\in \{1, 2, 3\}.$$
Finding the solution in the form 
$$ \vec{t}_i=\omega_{i1} c'(0)+\omega_{i2} c''(0) + \omega_{i3} c'(0)\times c'' (0)$$
we can express the coefficients $\omega_{i1}$ and $\omega_{i2}$. Since
$$ \omega_{i1} \left |c'(0)\right |^2 +\omega_{i2} \left<c'(0),c''(0)\right>=b_i (v) \ \ \textrm{and}\ \ \omega_{i1} \left<c''(0),c'(0)\right> +\omega_{i2} \left| c''(0)\right |^2=(b_i\circ c)'(0),$$
it follows by Cramer's rule that
$$\omega_{i1}=\frac{b_i(v)\left| c''(0)\right |^2 - \left<c'(0),c''(0)\right> (b_i\circ c)'}{| c'(0)\times c''(0)|^2}=\frac{b_i(v)}{\ |c'(0)|^2}-\frac{\langle c'(0),c''(0)\rangle}{|c'(0)|^5 \kappa^2(0)} \left(\frac{b_i\circ c}{|c'|}\right)'(0),$$
$$\omega_{i2}=\frac{\left| c'(0)\right |^2 (b_i\circ c)' - b_i(v) \left<c''(0),c'(0)\right>}{| c'(0)\times c''(0)|^2}=\frac{1}{|c'(0)|^3 \kappa^2(0)} \left(\frac{b_i\circ c}{|c'|}\right)'(0)$$
because of
$$| c'(0)\times c''(0)|^2=|c'(0)|^2|c''(0)|^2-\langle c'(0), c''(0) \rangle^2$$
and
$$\left(\frac{b_i\circ c}{|c'|}\right)'(0)=\frac{|c'(0)|^2 (b_i\circ c)'-b_i(v)\langle c'(0), c''(0) \rangle}{|c'(0)|^3}.$$
\end{proof}


According to the previous Lemma, the coefficients $\omega_{i1}$ and $\omega_{i2}$ are uniquely determined. They are related to the curvature of the integral curve of $C\times G$ (in a more general sense: the derivatives of the integral curve up to order two). To provide the arclength parametrization of the integral curve we are motivated to reformulate the argumentation by considering the normalized vector field of $C\times G$, where $G=\mathrm{grad} \ F$ is the Euclidean gradient on the indicatrix. 

\begin{Cor} Let $G=\mathrm{grad} \ F$ be the Euclidean gradient on the Finslerian spheres with respect to the compatible Riemannian metric, and consider the normalized vector field of $C\times G$. Taking the integral curve $c_0$ starting from a not vertical contact element $v$ in the tangent space $T_pM$, the solutions of \eqref{compgeneral} at the point $p$ are of the form
\begin{equation}
\label{compgeneralsolarclength}
 \begin{pmatrix} \vec{t}_1\\ \vec{t}_2\\ \vec{t}_3\end{pmatrix}=\begin{pmatrix} \omega_{11}^0 & \omega_{12}^0  & \omega_{13}^0 \\ \omega_{21}^0 & \omega_{22}^0  & \omega_{23}^0 \\ \omega_{31}^0 & \omega_{32}^0  & \omega_{33}^0 \end{pmatrix}\begin{pmatrix} \vec{T}(0)\\ \vec{N}(0)\\ \vec{B}(0)\end{pmatrix}, 
\end{equation}
where
$$\omega_{i1}^0=b_{i}^0(v), \ \omega_{i2}^0=\frac{(b_{i}^0\circ c_0)'(0)}{\kappa(0)}, \ b_{i}^0=\frac{b_i}{|C\times G|}=-2\frac{X_i^{h^*} F}{|C\times G|}, \quad  i\in \{1, 2, 3\}$$
and $(\vec{T}, \vec{N}, \vec{B})$ is the Frenet-Serret frame along $c_0$. 
\end{Cor}

\begin{proof}
Following the steps in the proof of the previous Lemma, the computations are straightforward.  
\end{proof}

Note that the determined part of (\ref{compgeneralsolarclength}) is lying in the osculating plane of the curve $c_0$ (or the curve $c$) at the starting parameter $0$. To clarify the contribution of the torsion of the integral curve to formula (\ref{compgeneralsolarclength}), we  present the following sufficient condition for the unicity of the solution. 

\begin{Cor} Let $G=\mathrm{grad} \ F$ be the Euclidean gradient on the Finslerian spheres with respect to the compatible Riemannian metric, and consider the normalized vector field of $C\times G$. If the integral curve $c_0$ starting from a not vertical contact element $v$ in the tangent space $T_pM$ has a non-vanishing torsion at the starting parameter, then the solutions of \eqref{compgeneral} at the point $p$ are of the form
\begin{equation}
\label{compgeneralsolarclengthtorsion}
 \begin{pmatrix} \vec{t}_1\\ \vec{t}_2\\ \vec{t}_3\end{pmatrix}=\begin{pmatrix} \omega_{11}^0 & \omega_{12}^0  & \omega_{13}^0 \\ \omega_{21}^0 & \omega_{22}^0  & \omega_{23}^0 \\ \omega_{31}^0 & \omega_{32}^0  & \omega_{33}^0 \end{pmatrix}\begin{pmatrix} \vec{T}(0)\\ \vec{N}(0)\\ \vec{B}(0)\end{pmatrix}, 
\end{equation}
where
$$\omega_{i1}^0=b_{i}^0(v), \ \omega_{i2}^0=\frac{(b_{i}^0\circ c_0)'(0)}{\kappa(0)}, \ \omega_{i3}^0=\frac{(b_{i}^0\circ c_0)''(0)+\kappa^2(0)\omega_{i1}^0-\kappa'(0) \omega_{i2}^0}{\kappa(0)\tau(0)},$$
$$b_{i}^0=\frac{b_i}{|C\times G|}=-2\frac{X_i^{h^*} F}{|C\times G|}, \quad  i\in \{1, 2, 3\}.$$
\end{Cor}

\begin{proof}
Evaluating (\ref{compgeneral}) along $c_0$, a simple differentiation gives 
$$\left<c_0'(0),\begin{pmatrix} t_i^1\\t_i^2\\t_i^3 \end{pmatrix}\right> =b_{i}^0(v), \ \left<c_0''(0),\begin{pmatrix} t_i^1\\t_i^2\\t_i^3 \end{pmatrix}\right> =(b_{i}^0\circ c_0)'(0), \ \left<c_0'''(0),\begin{pmatrix} t_i^1\\t_i^2\\t_i^3 \end{pmatrix}\right> =(b_{i}^0\circ c_0)''(0).$$
The last equation allows us to express the missing parameters $\omega_{13}^0$, $\omega_{23}^0$ and $\omega_{33}^0$ by substitu\-ting the solution of the form (\ref{compgeneralsolarclengthtorsion}). Note that $\left<c_0''(0),c_0'(0)\right>=0$ implies that
$$\left<c_0'''(0),c_0'(0)\right>=-\left<c_0''(0),c_0''(0)\right>=-\kappa^2(0)$$
and $\displaystyle{\left<c_0'''(0),c_0''(0)\right>=\frac{1}{2}\left<c_0'',c_0''\right>'(0)=\frac{1}{2} \left(\kappa^2\right)'(0)=\kappa(0)\kappa'(0)}$.
\end{proof}

\begin{Rem}{\emph{If $v\in T_pM$ is not vertical contact, then the Euclidean and the Finslerian indicatrices are transversally intersecting surfaces at $v$ and the integral curve $c_0$ (or $c$) is a parametrization of the intersection curve. The vanishing of the torsion means that the intersection is a plane curve, i.e.,  it must be a part of a Euclidean circle. Otherwise, the torsion of the intersection curve implies the possible values for the components of the torsion of the uniquely determined compatible linear connection.}}
\end{Rem}

\subsection{Compatible linear connections in the undetermined case (cf. Theorem \ref{thm:02})} Let $p\in M$ be a given point. 
If the solution sets of the compatibility equations \eqref{comp1}--\eqref{comp3} are 1-dimensional, they are parallel lines in $T_pM$, endowed with the inner product coming from the compatible Riemannian metric. Taking a not vertical contact element $v\in T_pM$, formula (\ref{compgeneralsolarclength}) shows that the general form of the solutions is
\begin{equation} \begin{array}{rcl}
\label{undet1}
\colvec{-T_{12}^3+T_{13}^2+T_{23}^1\\-2T_{13}^1\\2T_{12}^1\ }(p) &=&  \omega_{11}^0 \vec{T}(0)+ \omega_{12}^0 \vec{N}(0)+\omega_{13}^0 \vec{B}(0), \\
\\
\colvec{2T_{23}^2\\-T_{12}^3-T_{13}^2-T_{23}^1\\2T_{12}^2}(p) & =&  \omega_{21}^0 \vec{T}(0)+ \omega_{22}^0 \vec{N}(0)+\omega_{23}^0 \vec{B}(0), \\[20pt]
\\
\colvec{\ \ 2T_{23}^3\\-2T_{13}^3 \\T_{12}^3+T_{13}^2-T_{23}^1\ \ }(p) &=& \omega_{31}^0\vec{T}(0)+\omega_{32}^0 \vec{N}(0)+\omega_{33}^0 \vec{B}(0),
\end{array} \end{equation}
where
$$\omega_{i1}^0=b_{i}^0(v), \ \omega_{i2}^0=\frac{(b_{i}^0\circ c_0)'(0)}{\kappa(0)}, \ b_{i}^0=\frac{b_i}{|C\times G|}=-2\frac{X_i^{h^*} F}{|C\times G|}$$
and the parameters $\omega_{i3}^0$ can be arbitrarily chosen. The indicatrix at the point $p$ is a Euclidean surface of revolution with respect to the axis represented by the binormal vector $\vec{B}(0)$. For the torsion components with different indices we have a linear system of equations with an invertible coefficient matrix. Therefore formula (\ref{undet1}) gives all the possible torsion components at the point $p$. The corresponding compatible linear connections are determined by formula \eqref{torsion}.

\subsection{The extremal compatible linear connection in the undetermined case} To find the extremal one among the compatible linear connections we need to minimize the sum of the squares of the torsion components. Let us denote the solutions as
\begin{equation} \begin{array}{rcl}
\label{undet2}
\colvec{-T_{12}^3+T_{13}^2+T_{23}^1\\-2T_{13}^1\\2T_{12}^1\ }(p) &=& 2\colvec{P_1\\P_2\\P_3}+2s \colvec{D_1\\D_2\\D_3}, \\[20pt]
\\
\colvec{2T_{23}^2\\-T_{12}^3-T_{13}^2-T_{23}^1\\2T_{12}^2}(p) & =& 2\colvec{Q_1\\Q_2\\Q_3}+2t \colvec{D_1\\D_2\\D_3}, \\[20pt]
\\
\colvec{\ \ 2T_{23}^3\\-2T_{13}^3 \\T_{12}^3+T_{13}^2-T_{23}^1\ \ }(p) &=& 2\colvec{R_1\\R_2\\R_3}+2u \colvec{D_1\\D_2\\D_3},
\end{array} \end{equation}
where 
$$\colvec{D_1\\D_2\\D_3}=\vec{B}(0), \ \colvec{P_1\\P_2\\P_3}=\frac{1}{2}\left(b_{1}^0(v) \vec{T}(0)+\frac{(b_{1}^0\circ c_0)'(0)}{\kappa(0)} \vec{N}(0)\right),$$
$$\colvec{Q_1\\Q_2\\Q_3}=\frac{1}{2}\left(b_{2}^0(v) \vec{T}(0)+\frac{(b_{2}^0\circ c_0)'(0)}{\kappa(0)} \vec{N}(0)\right), \ \colvec{R_1\\R_2\\R_3}=\frac{1}{2}\left(b_{3}^0(v) \vec{T}(0)+\frac{(b_{3}^0\circ c_0)'(0)}{\kappa(0)} \vec{N}(0)\right),$$
$$s=\frac{1}{2}\omega_{13}^0, \ t=\frac{1}{2}\omega_{23}^0, \ u=\frac{1}{2}\omega_{33}^0$$
are the common directional vector, the determined parts in (\ref{undet1}) and the free parameters, respectively. Extra 2's are inserted to make the forthcoming formulas easy to review: 
\begin{equation}
\label{compconn1}
\begin{tabular}{lll}
$T_{13}^1(p)=-P_2-sD_2$, & $T_{23}^2(p)=Q_1+tD_1$,  & $T_{23}^3(p)=R_1+uD_1$, \\
&&\\
$T_{12}^1(p)=P_3+sD_3$,  & $T_{12}^2(p)=Q_3+tD_3$, & $T_{13}^3(p)=-R_2-uD_2$. \\  
\end{tabular}
\end{equation}
For components with different indices, we have the linear system
\[ \begin{pmatrix} -1 & \ \ 1 & \ \ 1 \\ -1 & -1 & -1 \\ \ \ 1 & \ \ 1 & -1 \end{pmatrix} \colvec{T_{12}^3\\T_{13}^2\\T_{23}^1}(p) = 2\colvec{P_1+sD_1 \\ Q_2+tD_2 \\ R_3+uD_3}.\] 
Here the coefficient matrix is invertible, and so the solution is 
\begin{equation}
\label{compconn2}
\colvec{T_{12}^3\\T_{13}^2\\T_{23}^1}(p) = \dfrac{1}{2} \begin{pmatrix} -1 & -1 & \ \ 0 \\ \ \ 1 & \ \ 0 & \ \ 1 \\ \ \ 0 & -1 & -1 \end{pmatrix} 2 \colvec{P_1+sD_1 \\ Q_2+tD_2 \\ R_3+uD_3} = \colvec{-P_1-Q_2-sD_1-tD_2\\ \ \ P_1+R_3+sD_1+uD_3\\-Q_2-R_3-tD_2-uD_3}.
\end{equation}
Using formulas \eqref{compconn1}--\eqref{compconn2}:
\[ \begin{gathered}
\norm{T_p}^2 = (T_{12}^1)^2(p)+(T_{12}^2)^2(p)+(T_{12}^3)^2(p)+(T_{13}^1)^2(p)+(T_{13}^2)^2(p)+(T_{13}^3)^2(p)+(T_{23}^1)^2(p)+\\ (T_{23}^2)^2(p)+(T_{23}^3)^2(p)= as^2+bt^2+cu^2+ds+et+fu+gst+hsu+itu+j,
\end{gathered} \]
where the coefficients are
$$a=2D_1^2+D_2^2+D_3^2, \ b = D_1^2+2D_2^2+D_3^2, \ c=D_1^2+D_2^2+2D_3^2,$$
$$d = 4D_1P_1+2D_1Q_2+2D_1R_3+2D_2P_2+2D_3P_3,$$
$$e=2D_1Q_1+2D_2P_1+4D_2Q_2+2D_2R_3+2D_3Q_3,$$ 
$$f=2D_1R_1+2D_2R_2+2D_3P_1+2D_3Q_2+4D_3R_3,$$
$$g=2D_1D_2, \ h=2D_1D_3,\ i=2D_2D_3,$$
$$j=2P_1^2+P_2^2+P_3^2+Q_1^2+2Q_2^2+Q_3^2+R_1^2+R_2^2+2R_3^2+2P_1Q_2+2P_1R_3+2Q_2R_3. $$

\begin{Thm} 
\label{thm:04} Using the notation
$$k = D_1^2+2D_2^2+2D_3^2, \ l=2D_1^2+D_2^2+2D_3^2, \ m=2D_1^2+2D_2^2+D_3^2, \ \norm{D}^2 = D_1^2+D_2^2+D_3^2,$$
the torsion components of the extremal compatible linear connection belong to the parameters
\[ s^0 = \dfrac{-2dk+eg+fh}{8 \norm{D}^4}, \ t^0 = \dfrac{dg-2el+fi}{8 \norm{D}^4}, \ u^0 = \dfrac{dh+ei-2fm}{8 \norm{D}^4} \]
in formulas \eqref{compconn1}--\eqref{compconn2}.
\end{Thm}

\begin{proof} Taking $\norm{T_p}^2$ as a function of the variables $s,t,u$, we have to find its global minimum point. At first, we are looking for the critical points:
$$\dfrac{\partial \norm{T_p}^2}{\partial s} = 2as+d+gt+hu = 0, \ \ \dfrac{\partial \norm{T_p}^2}{\partial t} = 2bt+e+gs+iu = 0, $$
$$\dfrac{\partial \norm{T_p}^2}{\partial u} =  2cu+f+hs+it = 0.$$
It is a linear system for the parameters $s,t,u$. Its matrix form is
\begin{equation}
\label{hessian}
\begin{pmatrix} 2a & g & h \\ g & 2b & i \\ h & i & 2c \end{pmatrix}\colvec{s\\t\\u}=\colvec{-d\\-e\\-f}.
\end{equation} 
In the second step we calculate the corner minors. To do this, we need the coefficient matrix, as the Hessian. Using MAPLE program, we find 
\[ \Delta_1 = 2a > 0, \ \Delta_2 = 4m\norm{D}^2 > 0, \ \Delta_3 = 16 \norm{D}^6 > 0.\]
Thus the coefficient matrix is invertible, so we have a unique triplet of stationary parameters, and the Hessian is positive definite. Therefore the critical point is a global minimizer. Using the inverse of the coefficient matrix of the linear system \eqref{hessian}, we find  
\[ \colvec{s^0\\t^0\\u^0} = H^{-1} \colvec{-d\\-e\\-f}, \ \ \textrm{where}\ \ H^{-1} =\begin{pmatrix} 2a & g & h \\ g & 2b & i \\ h & i & 2c \end{pmatrix}^{-1}= \dfrac{1}{8 \norm{D}^4} \begin{pmatrix} \ 2k & -g & -h \\ -g & \ 2l & -i \\ -h & -i & 2m  \end{pmatrix}.\]
\end{proof}

\subsection{The determined case (cf. Theorem \ref{thm:02})} Let $G=\mathrm{grad} \ F$ be the Euclidean gradient on the Finslerian spheres with respect to the compatible Riemannian metric, and consider the normalized vector field of $C\times G$. If $c_0$ is its integral curve starting from a not vertical contact element $v$ in the tangent space $T_pM$, then we can write the solution in the form (\ref{compgeneralsolarclength}). To determine the missing parameters $\omega_{i3}^0$ consider the action 
$$\Phi_v\colon \mathbb{R}\times T_pM \to \mathbb{R}, \ (t,w)\to \Phi_v(t,w):=F(e^{At}w)$$
of the one-parameter group of rotations around the binormal vector $\vec{B}(0)$ on the Finslerian spheres, where $A$ denotes the skew-symmetric matrix of the mapping $w\mapsto w\times \vec{B}(0)$.

\begin{Lem}
\label{lem:action}
If $\displaystyle{\dfrac{\partial \Phi_v}{\partial t}(0,w)=0}$ for any nonzero element $w\in T_pM$,
then the Finslerian spheres are Euclidean surfaces of revolution about the axis generated by the binormal vector $\vec{B}(0)$.
\end{Lem}

\begin{proof} Since $\Phi_v(t+s,w)=\Phi_v(t,e^{sA}w)$, it follows that 
$$\dfrac{\partial \Phi_v}{\partial t}(s,w)=\dfrac{\partial \Phi_v}{\partial t}(0,e^{sA}w).$$
Therefore the vanishing of the derivatives at $t=0$ implies that $\Phi_v$ is independent of the first variable, i.e.,  the Finslerian spheres are Euclidean surfaces of revolution about the axis generated by the binormal vector $\vec{B}(0)$.
\end{proof}

Since the Euclidean surfaces of revolution belong to the undetermined case, there is a non-zero element $w\in T_pM$ such that
$$\dfrac{\partial \Phi_v}{\partial t}(0,w)\neq 0$$
(see Lemma \ref{lemma:03}, Lemma \ref{lem:action} and Theorem \ref{thm:02}), thus we have that
$$\left<C_w\times G_w,\vec{B}(0)\right> =-\left<C_w\times \vec{B}(0), G_w\right>= \dfrac{\partial \Phi_v}{\partial t}(0,w) \neq 0.$$
To sum up, the uniquely determined solution must be of the form
\begin{equation}
\label{compgeneralsolarclengthfinal}
 \begin{pmatrix} \vec{t}_1\\ \vec{t}_2\\ \vec{t}_3\end{pmatrix}=\begin{pmatrix} \omega_{11}^0 & \omega_{12}^0  & \omega_{13}^0 \\ \omega_{21}^0 & \omega_{22}^0  & \omega_{23}^0 \\ \omega_{31}^0 & \omega_{32}^0  & \omega_{33}^0 \end{pmatrix}\begin{pmatrix} \vec{T}(0)\\ \vec{N}(0)\\ \vec{B}(0)\end{pmatrix}, 
\end{equation}
where  
$$\omega_{i1}^0=b_{i}^0(v), \ \omega_{i2}^0=\frac{(b_{i}^0\circ c_0)'(0)}{\kappa(0)},$$
$$\omega_{i3}^0=-\frac{1}{\left<C_w\times G_w,\vec{B}(0)\right>}\left(\omega_{i1}^0\left<C_w\times G_w,\vec{T}(0)\right>+\omega_{i2}^0\left<C_w\times G_w,\vec{N}(0)\right>+2 X_i^{h^*} F (w)\right),$$
$$b_{i}^0=\frac{b_i}{|C\times G|}=-2\frac{X_i^{h^*} F}{|C\times G|}, \quad i\in \{1 ,2, 3\}$$
and $(\vec{T}, \vec{N}, \vec{B})$ is the Frenet-Serret frame along the integral curve $c_0$ of the normalized vector field of $C\times G$ starting from a not vertical contact element $v\in T_pM$. Instead of the specification of the element $w$, we can integrate on the Euclidean unit sphere $\sigma_*$ with respect to the compatible Riemannian metric:
$$\omega_{i3}^0\left<C_w\times G_w,\vec{B}(0)\right>^2=$$
$$-\left<C_w\times G_w,\vec{B}(0)\right>\left(\omega_{i1}^0\left<C_w\times G_w,\vec{T}(0)\right>+\omega_{i2}^0\left<C_w\times G_w,\vec{N}(0)\right>+2 X_i^{h^*} F (w)\right),$$
i.e.,
$$\omega_{i3}^0\int_{\sigma_*} \left<C\times G,\vec{B}(0)\right>^2=$$
$$-\int_{\sigma_*}\left<C\times G,\vec{B}(0)\right>\left(\omega_{i1}^0\left<C\times G,\vec{T}(0)\right>+\omega_{i2}^0\left<C\times G,\vec{N}(0)\right>+2 X_i^{h^*} F \right),$$
where 
$$\int_{\sigma_*} \left<C\times G,\vec{B}(0)\right>^2\neq 0.$$

\section{Examples: 3-dimensional Randers spaces}

To illustrate the ideas and methods presented above, we consider a special class of Finsler manifolds.

\begin{Def} A Finsler manifold is called a \emph{Randers manifold} if the Finsler function has the form
\[ F(x,y) = \alpha(x,y) + \beta(x,y),  \]
where $\alpha$ is a norm coming from a Riemannian metric on $M$ given by $\alpha(x,y)=\sqrt{\alpha_{ij}(x)y^{i}y^{j}}$ in a local basis $\partial/\partial u^1, \ldots, \partial/\partial u^n$ and $\beta$ comes from a 1-form given by $\beta(x,y)=\beta_j(x)y^j$ such that $\alpha^{ij}\beta_i\beta_j< 1$. 
\end{Def}

Both the metric components of the Riemannian part and the components of the perturbating term are considered on the tangent manifold as composite functions $\alpha_{ij}(x)$ and $\beta_k(x)$, where $x=(x^1, \ldots, x^n)$.
It is well-known \cite{Vin1} that the Riemannian part is compatible to the Finsler function of a Randers manifold, i.e.,  $\gamma:=\alpha$ is a convenient choice for a compatible Riemannian metric. In what follows, we consider the Finsler function  $F=\alpha+\beta$ on a connected 3-dimensional manifold $M$. Let a point $p\in M$ be given and suppose\footnote{If $\beta_p={\bf 0}$ then $T_pM$ is a vertical contact tangent space and there is nothing to compute. Such a Randers manifold is a generalized Berwald manifold if and only if it  is a Riemannian manifold, i.e., the perturbating term is zero at each point of the manifold. Indeed, using the parallel transports induced by the compatible linear connection, the quadratic Finslerian indicatrix at a single point implies that the Finslerian indicatrix is quadratic at each point of the manifold.} that $\beta_p\neq {\bf 0}$. In order to make the computations easier, we choose local coordinates around $p$ such that 
\begin{itemize}
\item the coordinate vector fields form an orthonormal basis at $p$ with respect to the compatible Riemannian metric $\alpha$, i.e. $\alpha_{ij}(p)=\delta_{ij}$,
\item $\beta_1(p) = \beta_2(p)=0$ and $K:=\beta_3(p) \neq 0$, i.e., the coordinate vector fields $\partial/\partial u^1$ and $\partial/\partial u^2$ span the kernel of the linear functional  $\beta$ at the point $p$.  
\end{itemize}

Under these choices of the coordinate vector fields
\[F(x,y)=\sqrt{\delta_{ij}y^i y^j} +\beta_3(x)y^3 = \sqrt{(y^1)^2+(y^2)^2+(y^3)^2} + Ky^3. \]
The Finslerian spheres in $T_pM$ are given by the equations of type
\[ (y^1)^2+(y^2)^2+(1-K^2) \left(y^3+\dfrac{Kc}{1-K^2} \right)^2 = \dfrac{c^2}{1-K^2}. \]
For any given $c>0$ it is an ellipsoid with center $\left(0,0,-\dfrac{Kc}{1-K^2} \right)$ with axes belonging to the coordinate vector fields at $p\in M$. The Finslerian spheres are rotationally symmetric with respect to the axis of $y^3$. Therefore, by  Theorem \ref{thm:02}, we are in the undetermined case, i.e., there must be infinitely many compatible linear connections (if there are any), and the direction space of the solutions of the compatibility equations in $T_pM$ is given by the vector $(0,0,1)$. 

\subsection{The vector field $C\times G$ and its integral curves} The partial derivatives of $F$ with respect to the vectorial directions are
\[\dfrac{\partial F}{\partial y^k}(x,y)= \dfrac{y^k}{\sqrt{(y^1)^2+(y^2)^2+(y^3)^2}} + \delta^3_k K = \dfrac{y^k}{\alpha(x,y)} + \delta^3_k K \]
and the Euclidean gradient vector field of $F$ is
\[ G(x,y)=\dfrac{C}{\alpha}(x,y)+\left(0,\, 0,\, K \right)= \dfrac{1}{\alpha(x,y)}\left(y^1, \, y^2, \,y^3+K\alpha(x,y)\right), \]
where $C$ is the radial vector field. The cross product is
\[ C\times G = \dfrac{1}{\alpha(x,y)} \begin{vmatrix}
e_1 & e_2 & e_3 \\
y^1 & y^2 & y^3 \\
y^1 & y^2 & y^3+K\alpha(x,y)
\end{vmatrix} = K \colvec{\ \ y^2 \\ -y^1 \\ \ 0}. \]
Therefore, an element $v \in T_p M$ is vertical contact if and only if $v=(0,0,v^3)$, i.e.,  the vertical contact points are the elements of the $y^3$ coordinate axis. To simplify the formalism let us consider the vector field
\[\dfrac{C\times G}{K}= \colvec{\ \ y^2 \\ -y^1 \\ \ 0} \] 
instead of the normalized version of $C\times G$. Taking a not vertical contact element $v=(v^1,v^2,v^3) \in T_p M$, it is easy to see that the integral curve starting from $v$ and its derivatives can be given as 
\[ \arraycolsep=1pt \begin{array}{rclrccccc}
c(t) & = & \colvec{r \sin(t+t_0)\\r\cos(t+t_0)\\v^3}, &&
 c(0)&=&\colvec{r \sin t_0 \\r\cos t_0 \\v^3} &=& \colvec{v^1\\v^2\\v^3}, \\[25pt]
c'(t) & = & \colvec{\ \ r \cos(t+t_0)\\-r\sin(t+t_0)\\ \ 0}, &&
c'(0)&=&\colvec{\ \ r \cos t_0 \\-r\sin t_0 \\ \ 0} &=& \colvec{\ \ v^2\\-v^1\\ \ 0},\\[25pt]
c''(t) & = & \colvec{-r \sin(t+t_0)\\-r\cos(t+t_0)\\ \ 0}, &&
c''(0)&=&\colvec{-r \sin t_0\\-r\cos t_0\\ \ 0} &=& \colvec{-v^1\\-v^2\\ \ 0}. 
\end{array} \]

\subsection{The solution lines} Let us write the compatibility equations in the form 
\begin{equation}
\label{raderseq} \left< \dfrac{C\times G}{K},\colvec{t_i^1\\t_i^2 \\ t_i^3}\right> = -\dfrac{2}{K}X_i^{h^*} F,  \quad  i\in \{1, 2, 3\},
\end{equation} 
see \eqref{compgeneral}. To determine the directional space of the solutions at a not vertical contact $v$, we need to solve the system
\[ \arraycolsep=1pt \left.\begin{array}{rcl}
 \dotprod{c'(0), \vec{t}\ }  &=& 0 \\[2pt]
 \dotprod{c''(0), \vec{t}\ } &=& 0 \\
 \end{array}\right\} \Longleftrightarrow 
  \left.\begin{array}{rcrcl}
  v^2t^1 &-& v^1t^2 &=& 0 \\[2pt]
 -v^1t^1 &-& v^2t^2 &=& 0 \\
 \end{array}\right\} \Longleftrightarrow 
 \begin{pmatrix}
 \ \ v^2 & -v^1 \\ -v^1 & -v^2
 \end{pmatrix} \colvec{t^1\\t^2} = \colvec{0\\0},\]
see the proof of Lemma \ref{lemma:03}. It is clear that $t^3$ can be arbitrarily chosen and the direction space is the line generated by $(0,0,1)$. Since $\nabla^*$ is the L\'{e}vi-Civita connection of $\alpha$, we have that
\[ X_i^{h^*} F=X_i^{h^*}\alpha + X_i^{h^*} \beta = X_i^{h^*} \beta. \]
Furthermore,
\[ X_i^{h^*} \beta = \frac{\partial \beta_s y^s}{\partial x^i}-y^j {\Gamma}^{k*}_{ij}\circ \pi \frac{\partial \beta_s y^s}{\partial y^k} = \frac{\partial \beta_s}{\partial x^i}y^s-y^j {\Gamma}^{k*}_{ij}\circ \pi  \beta_k = y^j \left(  \frac{\partial \beta_j}{\partial x^i} - \Gamma^{3*}_{ij} \circ \pi \beta_3  \right) \]
and, consequently, 
\[ - \frac{2}{K} X_i^{h^*} F = 2 \left( {\Gamma}^{3*}_{ij}\circ \pi - \frac{1}{K} \frac{\partial \beta_j}{\partial x^i} \right)y^j. \]
Introducing the notation
\begin{equation} \label{C}
C_{j;i}:= \Gamma^{3*}_{ij}\circ \pi - \frac{1}{K} \frac{\partial \beta_j}{\partial x^i},
\end{equation}
\eqref{raderseq} takes the form
$$\left<\colvec{\ \ y^2 \\ -y^1 \\ 0}, \colvec{t_i^1\\t_i^2 \\ t_i^3}\right> = 2 C_{j;i}y^j, \ \ \textrm{i.e.},  \ \ y^2t_i^1 - y^1t_i^2 = 2 \left( C_{1;i}y^1+C_{2;i}y^2+C_{3;i}y^3\right), \quad   i\in \{1, 2, 3\}.$$
Differentiating these equations with respect to $y^1$, $y^2$ and $y^3$, the comparison of the coefficients gives that 
 the compatibility equations have solutions if and only if $C_{3;i}=0$. Furthermore, the determined parts are given by $t_i^1= 2C_{2;i}$ and $t_i^2= -2C_{1;i}$, respectively. 

\begin{Rem}
{\emph{The geometric meaning of conditions $C_{3;i}=0$ is that $\beta$ has a dual vector field of constant length with respect to the compatible Riemannian metric $\alpha$; for details see \cite{Vin1} and \cite{RandersGBM}.}}
\end{Rem}

\subsection{The torsion components of the compatible linear connections} They can be given by substituting 
\[\colvec{D_1\\D_2\\D_3}=\colvec{0\\0\\1}, \ 
\colvec{P_1\\P_2\\P_3}=\colvec{\ \ C_{2;1}\\-C_{1;1}\\0}, \
\colvec{Q_1\\Q_2\\Q_3}=\colvec{\ \ C_{2;2}\\-C_{1;2}\\0}, \
\colvec{R_1\\R_2\\R_3}=\colvec{\ \ C_{2;3}\\-C_{1;3}\\0} \]
in formulas \eqref{compconn1} and \eqref{compconn2}. The components 
$$T_{13}^1(p)=-P_2-sD_2=C_{1;1}, \ T_{23}^2(p)=Q_1+tD_1=C_{2;2}, \ T_{23}^3(p)=R_1+uD_1=C_{2;3},$$
$$T_{13}^3(p)=-R_2-uD_2=C_{1;3}, \ T_{12}^3(p)=-P_1-Q_2-sD_1-tD_2=C_{1;2}-C_{2;1}$$ 
are uniquely determined. Furthermore, 
$$T_{12}^1(p)=P_3+sD_3=s, \ T_{12}^2(p)=Q_3+tD_3=t, \ T_{13}^2(p)=P_1+R_3+sD_1+uD_3=C_{2;1}+u,$$
$$T_{23}^1(p)=-Q_2-R_3-tD_2-uD_3=C_{1;2}-u.$$
Using that  
$$d= 4D_1P_1+2D_1Q_2+2D_1R_3+2D_2P_2+2D_3P_3= 0,$$
$$e= 2D_1Q_1+2D_2P_1+4D_2Q_2+2D_2R_3+2D_3Q_3= 0,$$
$$f= 2D_1R_1+2D_2R_2+2D_3P_1+2D_3Q_2+4D_3R_3= 2D_3(P_1+Q_2)=2C_{2;1}-2C_{1;2},$$
$$g= 2D_1D_2 =0, \ h = 2D_1D_3 =0, \ i = 2D_2D_3 =0, \ k = D_1^2+2D_2^2+2D_3^2=2,$$
$$l=2D_1^2+D_2^2+2D_3^2=2,\ m = 2D_1^2+2D_2^2+D_3^2=1, \ \norm{D}^2 = D_1^2+D_2^2+D_3^2=1,$$
Theorem \ref{thm:04} shows that the coefficients of the extremal compatible linear connection belong to the parameters 
$$s^0 = \dfrac{-2dk+eg+fh}{8 \norm{D}^4} = 0, \ t^0 = \dfrac{dg-2el+fi}{8 \norm{D}^4} = 0,$$
$$u^0 = \dfrac{dh+ei-2fm}{8 \norm{D}^4} = \dfrac{4C_{2;1}-4C_{1;2}}{8}=\dfrac{C_{2;1}-C_{1;2}}{2}.$$

\section{Acknowledgement}

The authors would like to express their very great appreciation to Professor József Szilasi for his valuable and constructive suggestions.

\end{document}